\documentclass[12pt,a4paper,reqno]{amsart}
\usepackage{amsmath}
\usepackage{amsfonts}
\usepackage{amssymb}
\usepackage{amsthm}
\usepackage{enumerate}
\usepackage{centernot}
\usepackage{fullpage}
\usepackage{mathrsfs}
\usepackage{graphicx}
\usepackage{color}

\newcommand{\upperRomannumeral}[1]{\uppercase\expandafter{\romannumeral#1}}

\theoremstyle{plain}

  \newtheorem{proposition}[]{Proposition}
  \newtheorem{lemma}[]{Lemma}
  \newtheorem{theorem}[]{Theorem}
  
  \newtheorem{conjecture}[]{Conjecture}
  \newtheorem{remark}[]{Remark}

\title[High dimensional Ising model]{The effect of free boundary conditions \\ on the Ising model in high dimensions}

\author{Federico Camia}
\address{Division of Science, NYU Abu Dhabi, Saadiyat Island, Abu Dhabi, UAE \& Department of Mathematics, Faculty of Science, Vrije Universiteit Amsterdam, De Boelelaan 1111, 1081 HV Amsterdam, The Netherlands.}
\email{federico.camia@nyu.edu}
\author{Jianping Jiang}
\address{Beijing Institute of Mathematical Sciences and Applications, No.11 Yanqi Lake West Road, Beijing 101407, China.}
\email{jianpingjiang11@gmail.com}
\author{Charles M. Newman}
\address{Courant Institute of Mathematical Sciences, New York University,
	251 Mercer st, New York, NY 10012, USA, \& NYU-ECNU Institute of Mathematical
	Sciences at NYU Shanghai, 3663 Zhongshan Road North, Shanghai 200062, China.}
\email{newman@cims.nyu.edu}

\begin{document}
\begin{abstract}
We study the critical Ising model with free boundary conditions on finite domains in $\mathbb{Z}^d$ with $d\geq4$. Under the assumption, so far only proved completely for high~$d$, that the critical infinite volume two-point function is of order $|x-y|^{-(d-2)}$ for large $|x-y|$, we prove the same is valid on large finite cubes with free boundary conditions, as long as $x, y$ are not too close to the boundary. This confirms a numerical prediction in the physics literature by showing that the critical susceptibility in a finite domain of linear size $L$ with free boundary conditions is of order $L^2$ as $L\rightarrow\infty$. We also prove that the scaling limit of the near-critical (small external field) Ising magnetization field with free boundary conditions is Gaussian with the same covariance as the critical scaling limit, and thus the correlations do not decay exponentially. This is very different from the situation in low~$d$ or the expected behavior in high~$d$ with bulk boundary conditions.
\end{abstract}

\maketitle

\section{Introduction}
\subsection{Overview}\label{sec:overview}
It was proved in \cite{Aiz82,Fro82,ADC20} that the scaling limit of the critical Ising
model on $\mathbb{Z}^d$ when $d\geq4$ is trivial in the sense that any subsequential limit
of the Ising magnetization field is Gaussian. Let
$\langle \sigma_x\sigma_y\rangle_{\mathbb{Z}^d}$ be the critical two-point function on
$\mathbb{Z}^d$ with $d\geq 4$. It is known that
\begin{equation}\label{eq:twopointf}
	c|x-y|^{-(d-1)}\leq \langle \sigma_x\sigma_y\rangle_{\mathbb{Z}^d}\leq C|x-y|^{-(d-2)},~\forall x\neq y\in\mathbb{Z}^d,
\end{equation}
where $|x-y|$ stands for the Euclidean distance between $x$ and $y$. The upper bound in
\eqref{eq:twopointf} is proved in \cite{FSS76,Sok82} by the infrared bound, while the lower
bound can be proved by the Simon-Lieb inequality \cite{Sim80,Lie80}. A matching lower
bound for $\langle \sigma_x\sigma_y\rangle_{\mathbb{Z}^d}$ (i.e., with the
power $-(d-2)$) is so far only proved for $d$ sufficiently large, by the lace expansion
method --- see \cite{Sak07,Sak20}.

There are very few rigorous results about the high dimensional Ising model on finite
domains; however, there are many nonrigorous results (see, e.g.,
\cite{LM11,BKW12,LM14,LM16,ZGFDG18,FGZD20} and references therein). Despite the intense interest in the topic, in the physics literature there is still a debate about the behavior of the susceptibility in the Ising model on finite domains with free boundary conditions inside a critical window (see \cite{BKW12,LM14,LM16,ZGFDG18}), and that is one of the main motivations of this paper. Indeed, our first result (Theorem \ref{thm1}) implies that, assuming the decay rate \eqref{eq:assumption} for the critical two-point function in the thermodynamic limit, the critical susceptibility in a domain of linear size $n$ with free boundary conditions is of order $n^2$, confirming a prediction in \cite{LM11,LM14,ZGFDG18}---see Remark \ref{rem:debate} below for more details. We note that \eqref{eq:assumption} has been rigorously proved in sufficiently high dimensions and is believed to hold in any dimension $d \geq 4$.

We also study the magnetization field in systems with a vanishing external field. In two dimensions, it was shown \cite{CGN16,CJN20a} that choosing the external field to be proportional to an appropriate power of the lattice spacing and then sending the lattice spacing to zero (often called a near-critical scaling limit) produces a near-critical field which looks critical at short distances but exhibits exponential decay of correlations at long distances. This is true in the thermodynamic limit as well as in finite systems, regardless of the boundary conditions. Somewhat surprisingly, our second result (Theorem \ref{thm2}), combined with Theorem \ref{thm1}, shows that this is not the case in dimension $d>4$. While we expect that for $d>4$ the near-critical scaling limit does lead to a continuum field with exponential decay of correlations in the case of infinite-volume systems (see Conjecture \ref{con1}), our results indicate that the choice of free boundary conditions for a finite system with an external field prevents the system from getting near criticality in the scaling limit even when the external field is scaled to zero at the rate corresponding to the near-critical regime for the infinite-volume system.

We now review some rigorous results.  We will see that unlike for $d=2$,
correlations in high~$d$ depend on the boundary conditions very much. Let
$\langle \cdot\rangle_{\Lambda_n}^+$,  $\langle \cdot\rangle_{\Lambda_n}^p$,
and  $\langle \cdot\rangle_{\Lambda_n}^f$ denote the expectation for the critical Ising
model on $\Lambda_n:=[-n,n]^d$ with respect to all plus, periodic, and free boundary
conditions respectively. By Theorem~1 in Section 4.1 of \cite{Pap06}, one has
\begin{equation}\label{eq:tpp}
	\langle \sigma_0\sigma_x\rangle_{\Lambda_n}^p\geq \begin{cases} C_1|x|^{-(d-2)},&|x|\leq C_3n^{\frac{d}{2(d-2)}}\\
		C_2n^{-d/2},&|x|\geq C_3n^{\frac{d}{2(d-2)}}.
	\end{cases}
\end{equation}
It is conjectured in \cite{Pap06} that this piecewise function should be the
correct behavior (i.e., with the first $\geq$ replaced by $\approx$). Results similar to \eqref{eq:tpp} for simple random walk and weakly self-avoiding walk can be found in \cite{Sla20}. The main result in \cite{HHS19} is that under the assumption of \eqref{eq:assumption} below, for any $\epsilon>0$
\begin{equation}
	\langle \sigma_0\rangle_{\Lambda_n}^+\geq C_4n^{-1-\epsilon}.
\end{equation}
Therefore, by the GKS \cite{Gri67,KS68} and FKG inequalities, we have that for any $\epsilon>0$
\begin{equation}\label{eq:tp+}
	\langle \sigma_x\sigma_y\rangle_{\Lambda_n}^+\geq \langle \sigma_x\rangle_{\Lambda_n}^+\langle \sigma_y\rangle_{\Lambda_n}^+\geq \langle \sigma_0\rangle_{\Lambda_{2n}}^+\langle \sigma_0\rangle_{\Lambda_{2n}}^+\geq C_4^2(2n)^{-2-2\epsilon},~ \forall x,y\in\Lambda_n.
\end{equation}
For free boundary conditions, by the GKS inequalities and \eqref{eq:twopointf}, we have
\begin{equation}\label{eq:twopointfu}
	\langle\sigma_x\sigma_y\rangle_{\Lambda_{n}}^f\leq \langle\sigma_x\sigma_y\rangle_{\mathbb{Z}^d}\leq C|x-y|^{-(d-2)},~ \forall x,y\in\Lambda_n.
\end{equation}
One main aim of the current paper is to give a matching lower bound in \eqref{eq:twopointfu} when $x$ and $y$ are not too close to the boundary of $\Lambda_n$.

\subsection{Main results}
Let $\Lambda\subset\mathbb{Z}^d$ be finite. The classical Ising model on $\Lambda$ at inverse temperature $\beta$ with free boundary conditions and external field $H\in {\mathbb R}$ is defined by the probability measure $\mathbb{P}^f_{\Lambda,\beta,H}$ on $\{-1,+1\}^{\Lambda}$ such that for each $\sigma\in\{-1,+1\}^{\Lambda}$
\begin{equation}\label{eq:Isingdef}
	\mathbb{P}^f_{\Lambda,\beta,H}(\sigma)=\frac{\exp\left[\beta\sum_{\{x,y\}}\sigma_x\sigma_y+H\sum_{x\in\Lambda}\sigma_x\right]}{Z^f_{\Lambda,\beta,H}},
\end{equation}
where the first sum is over all  nearest-neighbor pairs in $\Lambda$, and
$Z^f_{\Lambda,\beta,H}$ is the partition function that makes \eqref{eq:Isingdef}
a probability measure.  In a more general formulation, the first sum is replaced by $\sum_{\{x,y\}}J_{xy}\sigma_x\sigma_y$, where $J_{xy}$ is the \emph{coupling constant} between $x$ and $y$. Later on we will consider models in which some of the coupling constants are set to zero.

Let $\langle\cdot\rangle^f_{\Lambda,\beta,H}$ denote the expectation
with respect to $\mathbb{P}^f_{\Lambda,\beta,H}$. It is well-known that
$\mathbb{P}^f_{\Lambda,\beta,H}$ converges 
to the infinite
volume measure $\mathbb{P}^f_{\mathbb{Z}^d,\beta,H}$ as $\Lambda\uparrow \mathbb{Z}^d$.
Let $\beta_c=\beta_c(\mathbb{Z}^d)$ be the inverse critical temperature.
It is also well-known that the limit $\mathbb{P}^f_{\mathbb{Z}^d,\beta,H}$
doesn't depend on the boundary conditions and the sequence $\Lambda\uparrow \mathbb{Z}^d$
if $\beta\leq\beta_c$ or $H>0$. In this paper, we focus on the critical and near-critical
Ising model, that is, we always fix $\beta=\beta_c$ unless otherwise stated; so we will typically
drop the $\beta$ subscript in $\mathbb{P}^f_{\mathbb{Z}^d,\beta,H}$ and
$\langle\cdot\rangle^f_{\Lambda,\beta,H}$ from now on. (Note however that Propositions \ref{prop:backbone} and \ref{prop:comp} below are valid for any $\beta$.)
We will also write
$\mathbb{P}^f_{\Lambda}$ and $\langle\cdot\rangle^f_{\Lambda}$ for
$\mathbb{P}^f_{\Lambda,\beta,H}$ and $\langle\cdot\rangle^f_{\Lambda,\beta,H}$,
respectively, if $\beta=\beta_c$ and $H=0$. It is expected that
for $d \geq 4$, the infinite volume two-point function behaves like
\begin{equation}\label{eq:assumption}
	c|x-y|^{-(d-2)}\leq \langle \sigma_x\sigma_y\rangle_{\mathbb{Z}^d}\leq C|x-y|^{-(d-2)},~\forall x\neq y\in\mathbb{Z}^d.
\end{equation}
As we mentioned previously, the upper bound in \eqref{eq:assumption} is proved in \cite{FSS76,Sok82} but the lower bound is only proved for sufficiently large $d$ in \cite{Sak07,Sak20}. In this paper, we use $c$ for constants which are usually ``small'' and $C$ for constants which are usually ``large;'' their actual values (depending only on $d$) may change from place to place. Our first main result is
\begin{theorem}\label{thm1}
	Suppose that $\eqref{eq:assumption}$ holds. Then there exist constants
$M>1$ and $c_1>0$ (depending only on $d$) such that the following holds uniformly in
$n\in\mathbb{N}$. For the Ising model on $\Lambda_{Mn}:=[-Mn,Mn]^d$ at
$\beta=\beta_c$ and $H=0$ with free boundary conditions, we have
	\begin{equation}\label{eq:mainresult}
		c_1|x-y|^{-(d-2)}\leq \langle \sigma_x\sigma_y\rangle_{\Lambda_{Mn}}^f\leq C|x-y|^{-(d-2)},~\forall x\neq y\in \Lambda_{n}.
	\end{equation}
\end{theorem}

\begin{remark}\label{rem:debate}
	The debate in the physics literature that we mentioned in Section \ref{sec:overview} is about the growth of the susceptibility
	\begin{equation}
		\chi_n(\beta):=\sum_{x,y\in\Lambda_n}\left[\langle\sigma_x\sigma_y\rangle_{\Lambda_n,\beta}^f-\langle\sigma_x\rangle_{\Lambda_n,\beta}^f\langle\sigma_y\rangle_{\Lambda_n,\beta}^f\right]/(2n+1)^d=\sum_{x,y\in\Lambda_n}\langle\sigma_x\sigma_y\rangle_{\Lambda_n,\beta}^f/(2n+1)^d.
	\end{equation}
It was argued in \cite{BKW12,ZGFDG18} that the maximum of $\chi_n(\beta)$ for $\beta$ in a neighborhood of $\beta_c$ grows like $n^{d/2}$ while in
\cite{LM11,LM14,LM16} it was argued that the order should be $n^2$.
Theorem~\ref{thm1} shows that $\chi_n(\beta_c)$ is of order $n^2$ under assumption~\eqref{eq:assumption}, and thus confirms a numerical prediction in \cite{LM11,LM14,ZGFDG18}.
That assumption is widely accepted in the
physics literature and rigorously proved for $d$ large, as noted earlier.
\end{remark}

\begin{remark}
	A similar result for critical Bernoulli percolation in high~$d$ can be
found in Theorem 1.2 of \cite{CH20}. Actually, Theorem 1.2 of \cite{CH20} is stronger than our Theorem~\ref{thm1}. That is, a result analogous to \eqref{eq:mainresult}
for critical Bernoulli percolation holds for all $M>1$, with $c_1$ and $C$ depending only on $M$ and $d$.
It is expected that such a stronger result should also hold for the Ising model.
\end{remark}

For the Ising model on $\mathbb{Z}^d$ with $d>4$ at $\beta=\beta_c$ and $H\downarrow0$,
it is expected that
\begin{equation}\label{eq:corrconj}
	\langle \sigma_x;\sigma_y\rangle_{\mathbb{Z}^d,H}:=	\langle \sigma_x\sigma_y\rangle_{\mathbb{Z}^d,H}-	\langle \sigma_x\rangle_{\mathbb{Z}^d,H}\langle \sigma_y\rangle_{\mathbb{Z}^d,H}\approx C_1\exp\left[-C_2H^{2/(d+2)}|x-y|\right].
\end{equation}
See (2.3) and (2.4) of \cite{BKW12} or (3.174) of \cite{Rao19} for the conjecture that the correlation length in the near-critical case (as $H\downarrow0$) behaves like $H^{-2/(d+2)}$. Now consider the Ising model with external field $H$ on the rescaled lattice $a\mathbb{Z}^d$ where $d>4$ and $a$ is small. From \eqref{eq:corrconj}, we expect to choose the external field $H=a^{(d+2)/2}h$ in
order to obtain a nontrivial scaling limit. Here is
such a conjecture about the near-critical scaling limit.
\begin{conjecture} \label{con1}
	Consider the near-critical Ising model on $a\mathbb{Z}^d$ with $d>4$, $\beta=\beta_c$ and $H=a^{(d+2)/2}h$ for some $h>0$. Then as $a\downarrow 0$,
	\begin{equation}
	\Phi^{a,h}:=a^{(d+2)/2}\sum_{x\in a\mathbb{Z}^d} [\sigma_x-\langle\sigma_x\rangle_{\mathbb{Z}^d,H}]\delta_x \Longrightarrow
\text{ massive Gaussian free field on }\mathbb{R}^d.
	\end{equation}
Here $\delta_x$ is a unit Dirac point measure at $x$ and $\Longrightarrow$
stands for convergence in distribution. The covariance function of the limiting field is proportional to the kernel of the operator $(-\Delta+m^2)^{-1}$ for some $m>0$, which decays exponentially.
\end{conjecture}
We remark that when $d=2$, a non-Gaussian scaling limit of $\Phi^{a,h}$ is
established in \cite{CGN16} and exponential decay of the
limiting field is proved in \cite{CJN20a,CJN20b}.

We next consider the near-critical Ising model on the finite domain $\Lambda^a_L:=[-L,L]^d\cap a\mathbb{Z}^d$. Let $\Phi^{a,h}_{\Lambda_L}$ be the magnetization field
\begin{equation}
	\Phi^{a,h}_{\Lambda_L}:=a^{(d+2)/2}\sum_{x\in \Lambda^a_L} \sigma_x\delta_x,
\end{equation}
where the superscript $h$ in $\Phi^{a,h}_{\Lambda_L}$ indicates that $\{\sigma_x\}_{x \in \Lambda^a_L}$ is distributed according to $\mathbb{P}^f_{\Lambda^a_L,H}$ with $H=ha^{(d+2)/2}$.
Aizenman's arguments in \cite{Aiz82} (see also \cite{ADC20}) should also imply that each
subsequential limit of $\Phi^{a,0}_{\Lambda_L}$ is Gaussian. When $d=2$, it is proved
in \cite{CJN20a} that the scaling limit of $\Phi^{a,h}_{\Lambda_L}$ with $h \neq 0$ already exhibits
exponential decay; loosely speaking, the truncated correlation of the limiting
field decays exponentially with a rate depending linearly on the distance.
Our next result is that there is no exponential decay in the near-critical scaling
limit if $d>4$,  at least when one first takes $a \to 0$ for fixed $L$ with
free boundary conditions and then takes $L \to \infty$.
The tightness of $\{\Phi^{a,0}_{\Lambda_L}: a>0\}$ can be established
by a combination of the methods in Appendix~A of~\cite{CGN15} and Theorem \ref{thm1}.
Even though it has not been proved that all subsequential limits are the same
(i.e., uniqueness of the limit), for the sake of simplicity, we will assume
\begin{equation}\label{eq:assumption2}
	\Phi^{a,0}_{\Lambda_L}\Longrightarrow \Phi_{\Lambda_L} \text{ as } a\downarrow 0
\end{equation}
under the topology of $\mathcal{H}^{-3}(\Lambda_L)$ (see \cite{CGN15} for the definition)
but cautious readers may use a subsequence instead.
\begin{theorem}\label{thm2}
	Suppose that \eqref{eq:assumption2} holds and assume that the field $\Phi_{\Lambda_L}$, obtained by taking the scaling limit of $\Phi^a_{\Lambda_L}$ with free boundary conditions on $\Lambda^a_L$,
	is a Gaussian field with covariance function $G_{\Lambda_L}$. Then for any $h>0$, with free boundary conditions,
	\begin{equation}
		\Phi^{a,h}_{\Lambda_L}\Longrightarrow \Phi^h_{\Lambda_L}
	\end{equation}
	under the topology of $\mathcal{H}^{-3}(\Lambda_L)$, where $\Phi^h_{\Lambda_L}$ is also
Gaussian with the same covariance function $G_{\Lambda_L}$ and mean given by
		\begin{equation}
			\langle \Phi^h_{\Lambda_L}(f)\rangle=h\int_{\Lambda_L}\int_{\Lambda_L}f(z)G_{\Lambda_L}(z,w)dzdw
		\end{equation}
	for each $f\in\mathcal{H}^3(\Lambda_L)$.
	
	Moreover, letting $M_{\Lambda_L} := \Phi_{\Lambda_L}(1[\Lambda_L])$ denote the total magnetization in $\Lambda_L$, where $1[\cdot]$ denotes the indicator function, we have that the Radon-Nikodym derivative of the law of $\Phi^h_{\Lambda_L}$ with respect to that of $\Phi_{\Lambda_L}$ is given by the Wick exponential of the mean-zero Gaussian random variable $hM_{\Lambda_L}$, i.e., $\exp\left[hM_{\Lambda_L} - \frac{h^2}{2}\emph{Var}(M_{\Lambda_L})\right]$.
\end{theorem}

\begin{remark}\label{rem:limitfield}
	By Theorem 1.3 of \cite{Sak07}, one has that for large $d$,
	\begin{equation}\label{eq:tpflimit}
		\lim_{a\downarrow 0} a^{-(d-2)}\langle\sigma_{x_a}\sigma_{y_a}\rangle_{a\mathbb{Z}^d}=C|x-y|^{-(d-2)},~\forall x\neq y\in\mathbb{R}^d,
	\end{equation}
where $x_a$ (resp., $y_a$) is a point in $a\mathbb{Z}^d$ that is closest to $x$ (resp., $y$). Combining this with the results in \cite{Aiz82}, we have that for large $d$,
\begin{equation}
	\Phi^{a,0}=a^{(d+2)/2}\sum_{x\in a\mathbb{Z}^d}\sigma_x\delta_x\Longrightarrow \text{ massless Gaussian free field}
\end{equation}
with covariance function $G_{\mathbb{Z}^d}(x,y)=C|x-y|^{-(d-2)}$.
Now our Theorem \ref{thm1} (more precisely, using \eqref{eq:tfin} below in its proof, together with \eqref{eq:twopointfu}) and Theorem~\ref{thm2} imply that for large $d$, as $L\uparrow\infty$, also
	\begin{equation}\label{eq:fflimit}
		\Phi^h_{\Lambda_L}(\cdot)-\langle \Phi^h_{\Lambda_L}(\cdot)\rangle\Longrightarrow \text{ massless Gaussian free field}
	\end{equation}
with the same covariance function $G_{\mathbb{Z}^d}$. It is expected that \eqref{eq:tpflimit}-\eqref{eq:fflimit} in fact hold for each $d>4$.
\end{remark}

The proof of Theorem \ref{thm1} relies on the random current representation of the
Ising model and its backbone representation, which will be introduced and discussed
briefly in Section~\ref{sec:rcm}. By the switching lemma, one can write the difference
of the two-point function in the infinite volume and in a finite domain as the
probability for existence of a long backbone connecting two ``close" points.
We then show that this probability is small by relating it to the infinite volume
correlation between two ``distant" points. Theorem \ref{thm2} follows from the
fact that the Radon-Nikodym derivative of the near-critical model with respect to
the critical model has the Gaussian form.

\subsection{Discussion of the main results}
As discussed in Remark \ref{rem:debate} above, Theorem \ref{thm1} confirms a numerical prediction on the behavior of the critical susceptibility in finite domains with free boundary conditions (see \cite{BKW12,LM14,LM16}).
Assuming the decay rate \eqref{eq:assumption} for the two-point function in the thermodynamic limit, Theorem \ref{thm1} implies that the critical susceptibility in a domain of linear size $n$ with free boundary conditions is of order $n^2$. The upper bound in \eqref{eq:assumption} is proved in \cite{FSS76,Sok82}; the lower bound is only proved for sufficiently large $d$ \cite{Sak07,Sak20} but is believed to hold for any $d \geq 4$.
	
Theorem \ref{thm2} may seem technical but it also contains an interesting and somewhat surprising message. Equipped with Theorem \ref{thm2}, one can consider the infinite-volume limit of the magnetization field $\Phi^h_{\Lambda_L}$ by sending $L \to \infty$. Let's assume that \eqref{eq:tpflimit} holds for some particular $d>4$ and that the scaling limit of the lattice magnetization in $\Lambda_L$ with free boundary conditions is a Gaussian field.
Then by Remark \ref{rem:limitfield}, the limit of $\Phi^h_{\Lambda_L}$ as $L \to \infty$ would have power law decay of correlations. On the contrary, as expressed in Conjecture~\ref{con1}, we expect that taking the near-critical scaling limit of the lattice magnetization on the whole space, $a{\mathbb Z}^d$, will lead to a continuum field with exponential decay of correlations. It appears that the infinite-volume limit and the scaling limit do not commute, and that interchanging them leads to very different results.
	
This should be contrasted with the situation in two dimensions where exponential decay already appears in finite volume, regardless of the boundary conditions, as shown in \cite{CJN20a}, and where interchanging the two limits leads to the same result, as shown in \cite{CGN16}. We note that the arguments in \cite{CGN16,CJN20a} make use of the RSW crossing probability bounds~\cite{CDCH16}, which are valid for two-dimensional critical systems but not in dimension higher than two. This explains why those arguments cannot be extended to higher dimensions, but the fact that exchanging the two limits leads to two very different results remains somewhat  mysterious.
	
What makes this a surprising fact is the heuristic view of a near-critical system as being characterized by a finite correlations length. According to this view, two parts of the system at distance much greater that the correlation length are essentially uncorrelated. In the case of a finite system of linear dimension much greater than the correlation length, the region around the origin should not feel the presence of the boundary. Hence, it seems natural to conclude that exponential decay should manifest itself already in finite systems, as is the case in two dimensions.
	
On the contrary, using Theorems \ref{thm1} and \ref{thm2} we see that, in dimension $d>4$, despite using the same scaling as in Conjecture \ref{con1}, a finite system with free boundary conditions is never really near-critical in the sense explained above. The exact mechanism by which the presence of a boundary somehow prevents the system from getting near criticality is unclear to us at the moment, and we think it is worth investigating.

\section{Random current representation and its properties}\label{sec:rcm}
In this section, we briefly introduce the random current representation for the Ising model without external field (i.e., $H=0$) and list some of its properties that will be used in the proof of Theorem \ref{thm1}. We refer to \cite{Aiz82,AF86,ADCS15} for more details about this representation.

\subsection{The random current representation}
Let $G=(V,E)$ be a finite subgraph of the nearest neighbor graph on $\mathbb{Z}^d$,
where $V$ is the set of vertices and $E$ is the set of edges. A \emph{current}
$\mathbf{n}$ on $G$ is a function from $E$ to $\mathbb{N}_0:=\mathbb{N}\cup \{0\}$.
A \emph{source} of
$\mathbf{n}=\{\mathbf{n}_{xy}: \{x,y\}\in E\}$ is a vertex $x\in V$ at which
$\sum_{y:\{x,y\}\in E}\mathbf{n}_{xy}$ is odd. The set of sources of $\mathbf{n}$
is denoted by $\partial\mathbf{n}$. Let $G_1=(V_1,E_1)$ be a subgraph of $G$. For any $x,y\in V_1$, we write $x\overset{\mathbf{n}}{\longleftrightarrow} y$ in $G_1$ for the event that there is a path $x=v_0,v_1,\dots,v_n=y$ such that $v_iv_{i+1}\in E_1$ and $\mathbf{n}_{v_iv_{i+1}}>0$ for each $0\leq i<n$. For every fixed $\mathbf{n}\in\mathbb{N}_0^E$, its weight is defined by
\begin{equation}
	w(\mathbf{n}):=\prod_{\{x,y\}\in G}\frac{\beta^{\mathbf{n}_{xy}}}{\mathbf{n}_{xy}!}.
\end{equation}
One important property of the random current representation is that, for any $A\subseteq V$,
\begin{equation}\label{eq:corrcr}
	\langle\sigma_A\rangle_{G,\beta}^f:=\big\langle\prod_{x\in A}\sigma_x\big\rangle_{G,\beta}^f=\frac{\sum_{\partial\mathbf{n}=A}w(\mathbf{n})}{\sum_{\partial\mathbf{n}=\emptyset}w(\mathbf{n})},
\end{equation}
where the sum in the numerator (resp., denominator) is over all $\mathbf{n}\in\mathbb{N}_0^E$
such that $\partial\mathbf{n}=A$ (resp., $\partial\mathbf{n}=\emptyset$). The following switching lemma turns out to be very useful in dealing with Ising correlations.
\begin{lemma}[Switching Lemma]\label{lem:switching}
	Suppose that $G_1=(V_1,E_1)$  is a subgraph of $G=(V,E)$. Then for any $x,y\in V_1$ and $A\subseteq V$, and any function $F:\mathbb{N}_0^E\rightarrow \mathbb{R}$, we have
	\begin{align}
		&\sum_{\substack{\mathbf{n}\in \mathbb{N}_0^E: \partial\mathbf{n}=A\\\mathbf{m}\in \mathbb{N}_0^{E_1}: \partial\mathbf{m}=\{x,y\}}}w(\mathbf{n})w(\mathbf{m})F(\mathbf{n}+\mathbf{m})\\
		&\qquad=\sum_{\substack{\mathbf{n}\in \mathbb{N}_0^E: \partial\mathbf{n}=A\Delta\{x,y\}\\\mathbf{m}\in \mathbb{N}_0^{E_1}: \partial\mathbf{m}=\emptyset}}w(\mathbf{n})w(\mathbf{m})F(\mathbf{n}+\mathbf{m})1\left[x\overset{\mathbf{n}+\mathbf{m}}{\longleftrightarrow}y \text{ in }G_1 \right],
	\end{align}
where $A\Delta B:=(A\setminus B)\cup (B\setminus A)$ is the symmetric difference and $1[\cdot]$ is the indicator function.
\end{lemma}
\begin{proof}
	See Lemma 2.2 of \cite{ADCS15}.
\end{proof}

\subsection{The backbone representation of random currents}
Each unoriented nearest neighbor edge $\{x,y\}$ of $\mathbb{Z}^d$ corresponds to
two oriented edges: one from $x$ to $y$ (denoted by $(x,y)$) and the other
from $y$ to $x$ (denoted by $(y,x)$).  We fix an arbitrary order on all oriented
edges in $\mathbb{Z}^d$.  To each oriented edge $(x,y)$,  we associate a set of
\emph{canceled edges} consisting of $\{x,y\}$ and all unoriented edges $\{x,z\}$
such that $(x,z)$ appears earlier than $(x,y)$ in our fixed order. A sequence of
oriented edges is said to be \emph{consistent} if no edge of the sequence corresponds to an
unoriented edge cancelled by a previous edge. If $\omega$ is a consistent sequence
of oriented edges, we denote by $\tilde{\omega}$ the set of all unoriented edges its edges cancel.

The \emph{backbone of $\mathbf{n}$} with $\partial\mathbf{n}=\{x,y\}$, denoted by $\omega(\mathbf{n})$, is the (oriented) edge self-avoiding path from $x$ to $y$ passing only through edges $e$ with $\mathbf{n}_e$ odd which is minimal under the lexicographical order on paths induced by the fixed order. More generally, a \emph{backbone} is an edge self-avoiding path which is consistent. We only consider backbones between two points (called \emph{sources}) in this paper and the reader may refer to \cite{AF86} for the
more general definition. For each fixed backbone $\omega$ with sources
$\partial\omega=\{x,y\}$, we define
\begin{equation}\label{eq:rhodef}
	\rho_G(\omega):=\frac{\sum_{\partial\mathbf{n}=\{x,y\}}w(\mathbf{n})1[\omega(\mathbf{n})=\omega]}{\sum_{\partial\mathbf{n}=\emptyset}w(\mathbf{n})}.
\end{equation}

The following properties of the backbone representation,
based on results from \cite{AF86, ADCS15},
will be very important to the proof of Theorem \ref{thm1}.
\begin{proposition}\label{prop:backbone}
Let $\Lambda\subset \mathbb{Z}^d$ be finite. (We will also denote by $\Lambda$ the graph with vertex set $\Lambda$ and edges that are nearest-neighbor pairs of vertices in $\Lambda$.)
	\begin{enumerate}[(a)]
		\item $\langle\sigma_x\sigma_y\rangle_{\Lambda,\beta}^f=\sum_{\partial\omega=\{x,y\}}\rho_{\Lambda}(\omega)$.
		\item If $\omega_1\circ\omega_2$ denotes the concatenation of
	two backbones $\omega_1$ and $\omega_2$ and $\omega_1\circ\omega_2$ is consistent, then we have
		\begin{equation}
			\rho_{\Lambda}(\omega_1\circ\omega_2)=\rho_{\Lambda}(\omega_1)\rho_{\Lambda\setminus\tilde{\omega}_1}(\omega_2),
		\end{equation}
	    where $\Lambda\setminus\tilde{\omega}_1$ means that the coupling constant for each edge in $\tilde{\omega}_1$ is set to zero.
		\item For a fixed backbone $\omega$ with finitely many edges, we have that the limit
		\begin{equation}
			\lim_{\Lambda\uparrow\mathbb{Z}^d}\rho_{\Lambda}(\omega)
		\end{equation}
	exists; it will be denoted by $\rho_{\mathbb{Z}^d}(\omega)$.
	\end{enumerate}
\end{proposition}
\begin{proof}
	Part (a) follows from (4.2) of \cite{AF86} and (b) follows from (4.7) of \cite{AF86}. From (4.5) of~\cite{AF86}, for a fixed backbone $\omega$, we have
	\begin{align}
		\rho_{\Lambda}(\omega)&=\prod_{e\in\omega}\tanh(\beta)\frac{\sum_{\partial\mathbf{n}=\emptyset}w(\mathbf{n})1[\mathbf{n}_e\text{ is even for each } e\in\tilde{\omega}]}{Z^f_{\Lambda,\beta}}\\
		&=\prod_{e\in\omega}\tanh(\beta)\frac{\sum_{\partial\mathbf{n}=\emptyset}w(\mathbf{n})1[\mathbf{n}_e\text{ is even for each } e\in\tilde{\omega}]}{\sum_{\partial\mathbf{n}=\emptyset}w(\mathbf{n})}.\label{eq:rhofull}
	\end{align}
Here the unoriented edge $e\in\omega$ means that one of the two oriented edges corresponding
to $e$ is in $\omega$.  The fraction in \eqref{eq:rhofull} has a limit as
$\Lambda\uparrow\mathbb{Z}^d$ by Theorem 2.3 of \cite{ADCS15}. This completes the
proof of part (c) in the proposition.
\end{proof}

\section{Proof of the main results}
In this section, we prove Theorems \ref{thm1} and \ref{thm2}. We first prove a comparison
result for correlations which will be a key step in the proof of Theorem \ref{thm1}.
Here is some notation. For $x\in\mathbb{Z}^d$, let $\bar{x}$ be the reflection of $x$ with
respect to the plane $x_1=0$, i.e.,
\begin{equation}
	\bar{x}=(-x_1,x_2,\dots,x_d) \text{ for each } x=(x_1,x_2,\dots,x_d)\in\mathbb{Z}^d.
\end{equation}
 If $A$ is a family of edges in $\mathbb{Z}^d$, let $\bar{A}$ be the reflection of $A$ with respect to the plane $x_1=0$, i.e.,
\begin{equation}
	\bar{A}=\{\{\bar{v},\bar{w}\}: \{v,w\}\in A\}.
\end{equation}

\begin{proposition}\label{prop:comp}
	For $K,L\in\mathbb{N}$, let $D:=[-K,K]\times[-L,L]^{d-1}$. Let $A$ be a  collection of edges in $D$ whose endpoints have nonnegative first coordinate, that is,
	\begin{equation}
		\{v,w\}\in A,  \text{ iff } v, w\in D \text{ with } v_1\geq 0 \text{ and } w_1\geq0.
	\end{equation}
Then for any $u\in D$ with $u_1=0$ and any $y\in D$ with $y_1\geq 0$, we have
	\begin{equation}
		\langle \sigma_u\sigma_y\rangle_{D\setminus \bar{A},\beta}^f\geq \langle \sigma_u\sigma_{\bar{y}}\rangle_{D\setminus \bar{A},\beta}^f,
	\end{equation}
where $\langle\cdot\rangle_{D\setminus \bar{A},\beta}^f$ means that the coupling constant for each edge in $\bar{A}$ is set to $0$ while the coupling constants for all other edges in $D$ are still $1$,
or equivalently we consider the graph with vertex set $D$ and edges between nearest-neighbor vertices but with the edges in $\bar{A}$ removed. (see Figure \ref{fig:cor}).
\end{proposition}
\begin{figure}
	\begin{center}
		\includegraphics[width=0.8\textwidth]{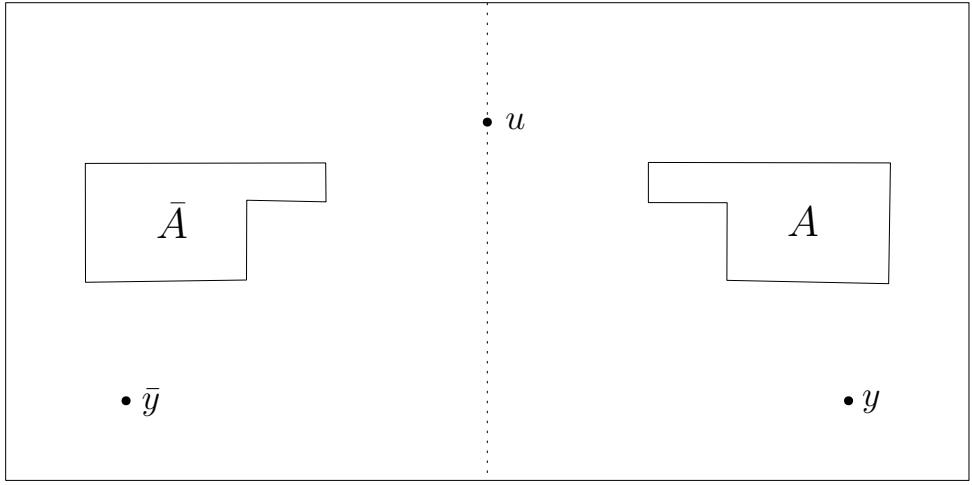}
		\caption{An illustration of the notations in Proposition \ref{prop:comp}.}\label{fig:cor}
	\end{center}
\end{figure}
\begin{proof}
	Our argument is inspired by the proof of Theorem 1 of \cite{MMS77}.
Let $D_+$ be the set of points $x\in D$ with $x_1\geq0$.
We will slightly abuse notation below by sometimes using $D$ (or $D_+$) to denote the set of
edges both of whose vertices are in $D$ (or $D_+$).
For each $x\in D_+$, we define
	\begin{equation}
		s_x=(\sigma_x+\sigma_{\bar{x}})/2,~	t_x=(\sigma_x-\sigma_{\bar{x}})/2.
	\end{equation}
	Then the Ising Hamiltonian in $D\setminus\bar{A}$ can be written as
	\begin{align}
		-H_{D\setminus\bar{A}}(\sigma)&=\sum_{\{x,y\}\in D\setminus\bar{A}}\sigma_x\sigma_y\\
		&=\sum_{\substack{\{x,y\}\in D_+\setminus A\\\{x_1,y_1\}\neq \{0,0\}}}(\sigma_x\sigma_y+\sigma_{\bar{x}}\sigma_{\bar{y}})+\sum_{\substack{\{x,y\}\in D\setminus A\\x_1=y_1=0}}\sigma_x\sigma_y+\sum_{\{x,y\}\in A\setminus \bar{A}}\sigma_x\sigma_y\\
		&=2\sum_{\substack{\{x,y\}\in D_+\setminus A\\ \{x_1,y_1\}\neq \{0,0\}}}(s_xs_y+t_xt_y)+\sum_{\substack{\{x,y\}\in D\setminus A\\x_1=y_1=0}}(s_x+t_x)(s_y+t_y)\\
		&\quad +\sum_{\{x,y\}\in A\setminus \bar{A}}(s_x+t_x)(s_y+t_y).
	\end{align}
Therefore, $H_{D\setminus\bar{A}}(\sigma)$ can be viewed as the Hamiltonian of a
spin system $\{(s_x,t_x):x\in D_+\}$ (with some $s_xt_y$ and $s_yt_x$ terms). It is clear
that this spin system is ferromagnetic and invariant under spin flipping, so by
the GKS inequalities \cite{Gri67,KS68,Gin70}, we have
	\begin{equation}
		\langle s_Bt_C \rangle_{D_+}\geq 0, \text{ for any } B,C\subseteq D_+.
	\end{equation}
	Setting $B=u$ and $C=y$ in the last displayed equation, we complete the proof the proposition.
\end{proof}
We are now ready to prove Theorem \ref{thm1}.
\begin{proof}[Proof of Theorem \ref{thm1}]
	The  upper bound in \eqref{eq:mainresult} follows from the GKS inequalities and our assumption \eqref{eq:assumption}. So it remains to prove the lower bound in \eqref{eq:mainresult}.
	For any finite $\Lambda\subseteq\mathbb{Z}^d$ such that $\Lambda_{4Mn}\subseteq\Lambda$, we have by the random current representation (see \eqref{eq:corrcr}) that
	\begin{align}
		\langle\sigma_x\sigma_y\rangle_{\Lambda}^f-\langle\sigma_x\sigma_y\rangle_{\Lambda_{Mn}}^f&=\frac{\sum_{\partial\mathbf{n}=\{x,y\}}w(\mathbf{n})}{\sum_{\partial\mathbf{n}=\emptyset}w(\mathbf{n})}
		-\frac{\sum_{\partial\mathbf{m}=\{x,y\}}w(\mathbf{m})}{\sum_{\partial\mathbf{m}=\emptyset}w(\mathbf{m})}\\
		&=\frac{\sum_{\partial\mathbf{n}=\{x,y\},\partial\mathbf{m}=\emptyset}w(\mathbf{n})w(\mathbf{m})-\sum_{\partial\mathbf{n}=\emptyset,\partial\mathbf{m}=\{x,y\}}w(\mathbf{n})w(\mathbf{m})}{\sum_{\partial\mathbf{n}=\emptyset,\partial\mathbf{m}=\emptyset}w(\mathbf{n})w(\mathbf{m})},\label{eq:cordif}
	\end{align}
where $\sum$ above and below represents a sum over some
$\mathbf{n}\in\mathbb{N}_0^{\Lambda}$ or some $\mathbf{m}\in\mathbb{N}_0^{\Lambda_{Mn}}$
or a double sum over both. By Lemma~\ref{lem:switching}, we have
	\begin{equation}\label{eq:switch}
		\sum_{\partial\mathbf{n}=\emptyset,\partial\mathbf{m}=\{x,y\}}w(\mathbf{n})w(\mathbf{m})=\sum_{\partial\mathbf{n}=\{x,y\},\partial\mathbf{m}=\emptyset}w(\mathbf{n})w(\mathbf{m})1[x\overset{\mathbf{n}+\mathbf{m}}{\longleftrightarrow} y \text{ in }\Lambda_{Mn}] .
	\end{equation}
Plugging \eqref{eq:switch} into \eqref{eq:cordif}, we get
	\begin{equation}
		\langle\sigma_x\sigma_y\rangle_{\Lambda}^f-\langle\sigma_x\sigma_y\rangle_{\Lambda_{Mn}}^f=\frac{\sum_{\partial\mathbf{n}=\{x,y\},\partial\mathbf{m}=\emptyset}w(\mathbf{n})w(\mathbf{m})1[x\overset{\mathbf{n}+\mathbf{m}}{\centernot\longleftrightarrow y} \text{ in }\Lambda_{Mn}]}{\sum_{\partial\mathbf{n}=\emptyset,\partial\mathbf{m}=\emptyset}w(\mathbf{n})w(\mathbf{m})}.
	\end{equation}
	Since $\{x\overset{\mathbf{n}+\mathbf{m}}{\centernot\longleftrightarrow y} \text{ in }\Lambda_{Mn}\}\subseteq\{x\overset{\mathbf{n}}{\centernot\longleftrightarrow} y \text{ in }\Lambda_{Mn}\}$, we have
	\begin{equation}\label{eq:cordif1}
		\langle\sigma_x\sigma_y\rangle_{\Lambda}^f-\langle\sigma_x\sigma_y\rangle_{\Lambda_{Mn}}^f\leq\frac{\sum_{\partial\mathbf{n}=\{x,y\}}w(\mathbf{n})1[x\overset{\mathbf{n}}{\centernot\longleftrightarrow} y \text{ in }\Lambda_{Mn}])}{\sum_{\partial\mathbf{n}=\emptyset}w(\mathbf{n})}.
	\end{equation}
	
	Using the backbone representation and the $\rho$ function defined in \eqref{eq:rhodef}, we have
	\begin{equation}\label{eq:bb1}
		\frac{\sum_{\partial\mathbf{n}=\{x,y\}}w(\mathbf{n})1[x\overset{\mathbf{n}}{\centernot\longleftrightarrow} y \text{ in }\Lambda_{Mn}])}{\sum_{\partial\mathbf{n}=\emptyset}w(\mathbf{n})}=\sum_{\substack{\partial\omega=\{x,y\}\\\omega\cap\Lambda_{Mn}^c\neq\emptyset}}\rho_{\Lambda}(\omega),
	\end{equation}
where $\omega\cap\Lambda_{Mn}^c\neq\emptyset$ means that $\omega$ should use at least one edge in the complement of $\Lambda_{Mn}$.

Let $\partial\Lambda_{Mn}$ be the boundary of $\Lambda_{Mn}$, i.e.,
\begin{equation}
	\partial\Lambda_{Mn}:=\{z\in\Lambda_{Mn}: \exists\text{ nearest-neighbor edge }\{z,w\} \text{ such that } w\notin\Lambda_{Mn} \}.
\end{equation}

For each $\omega=\{v_i:0\leq i\leq I\}$ satisfying $\omega\cap\Lambda_{Mn}^c\neq\emptyset$, let $\tau$ be the first $i$ such that $v_i\in\partial\Lambda_{Mn}$, and write $\omega_1$ for the walk $\{v_i:0\leq i\leq \tau\}$ and $\omega_2$ for the rest of the walk $\omega$ (i.e. $\{v_i:\tau\leq i\leq I\}$); then it is clear that $\omega_1$ is a backbone from $x$ to some $u\in\partial\Lambda_{Mn}$ which does not contain any other vertex in $\partial\Lambda_{Mn}$, and $\omega_2$ is a backbone from $u$ to $y$ satisfying $\omega_2\cap\tilde{\omega}_1=\emptyset$. It is easy to see that such a decomposition map $\omega\mapsto (\omega_1,\omega_2)$ is bijective. Therefore, we can write
\begin{equation}
	\sum_{\substack{\partial\omega=\{x,y\}\\\omega\cap\Lambda_{Mn}^c\neq\emptyset}}\rho_{\Lambda}(\omega)=\sum_{u\in\partial\Lambda_{Mn}}\sideset{}{'}\sum_{\partial\omega_1=\{x,u\}}\sum_{\substack{\partial\omega_2=\{u,y\}\\\omega_2\cap\tilde{\omega}_1=\emptyset}}\rho_{\Lambda}(\omega_1\circ\omega_2),
\end{equation}
 where $\sideset{}{'}\sum_{\partial\omega_1=\{x,u\}}$ denotes the sum over all backbones $\omega_1$ from $x$ to $u$ which first hit $\partial\Lambda_{Mn}$ at $u$.

Now properties (b) and (a) in Proposition \ref{prop:backbone} applied to the last
displayed equation give
\begin{align}
	\sum_{\substack{\partial\omega=\{x,y\}\\\omega\cap\Lambda_{Mn}^c\neq\emptyset}}\rho_{\Lambda}(\omega)&=\sum_{u\in\partial\Lambda_{Mn}}\sideset{}{'}\sum_{\partial\omega_1=\{x,u\}}\sum_{\substack{\partial\omega_2=\{u,y\}\\\omega_2\cap\tilde{\omega}_1=\emptyset}}\rho_{\Lambda}(\omega_1)\rho_{\Lambda\setminus\tilde{\omega}_1}(\omega_2)\\
	&=\sum_{u\in\partial\Lambda_{Mn}}\sideset{}{'}\sum_{\partial\omega_1=\{x,u\}}\rho_{\Lambda}(\omega_1)\langle\sigma_u\sigma_y\rangle_{\Lambda\setminus\tilde{\omega}_1}^f .
\end{align}
If we fix $M$ and $n$ and let $\Lambda\uparrow\mathbb{Z}^d$, we obtain by property (c) of
Proposition \ref{prop:backbone} and the GKS inequalities (to get the monotonicity of $\langle\cdot\rangle_{\Lambda\setminus\tilde{\omega}_1}^f$ in $\Lambda$) that
\begin{equation}\label{eq:sumrholimit}
	\lim_{\Lambda\uparrow\mathbb{Z}^d}\sum_{\substack{\partial\omega=\{x,y\}\\\omega\cap\Lambda_{Mn}^c\neq\emptyset}}\rho_{\Lambda}(\omega)=\sum_{u\in\partial\Lambda_{Mn}}\sideset{}{'}\sum_{\partial\omega_1=\{x,u\}}\rho_{\mathbb{Z}^d}(\omega_1)\langle\sigma_u\sigma_y\rangle_{\mathbb{Z}^d\setminus\tilde{\omega}_1}.
\end{equation}

Note that $\partial\Lambda_{Mn}$ has $2d$ faces.
We order the faces of $\partial\Lambda_{Mn}$ and denote them by $F_i, 1\leq i\leq 2d$, and define $y^i, 1\leq i\leq 2d$, to be the image of $y$ under the reflection with
respect to the plane containing $F_i$. Without loss of
generality, we can assume that $y^1$ is the image of $y$ under the
reflection with respect to the plane $\{z\in\mathbb{R}^d:z_1=Mn\}$. We choose a
sequence of domains $D_j:=[-j+Mn,j+Mn]\times[-j,j]^{d-1}$.
Then Proposition \ref{prop:comp} implies that, for all large $j$ and all $u$ such that $u_1=Mn$,
\begin{equation}
	\langle\sigma_u\sigma_y\rangle_{D_j\setminus\tilde{\omega}_1}^f\leq\langle\sigma_u\sigma_{y^1}\rangle_{D_j\setminus\tilde{\omega}_1}^f.
\end{equation}
Sending $j$ to infinity in the last displayed inequality, we get
\begin{equation}\label{eq:corcomp}
	\langle\sigma_u\sigma_y\rangle_{\mathbb{Z}^d\setminus\tilde{\omega}_1}\leq\langle\sigma_u\sigma_{y^1}\rangle_{\mathbb{Z}^d\setminus\tilde{\omega}_1}.
\end{equation}

Note that in \eqref{eq:sumrholimit} and \eqref{eq:corcomp}, we have implicitly used
the fact that the limit $\langle\cdot\rangle_{\mathbb{Z}^d\setminus\tilde{\omega}_1}$  is independent of the sequence. Depending on which face contains $u$,
there are actually $2d$ inequalities similar to \eqref{eq:corcomp}.
Applying all these inequalities to \eqref{eq:sumrholimit}, we have
\begin{align}
\lim_{\Lambda\uparrow\mathbb{Z}^d}\sum_{\substack{\partial\omega=\{x,y\}\\\omega\cap\Lambda_{Mn}^c\neq\emptyset}}\rho_{\Lambda}(\omega)&\leq\sum_{i=1}^{2d}\sum_{u\in F_i}\sideset{}{'}\sum_{\partial\omega_1=\{x,u\}}\rho_{\mathbb{Z}^d}(\omega_1)\langle\sigma_u\sigma_{y^i}\rangle_{\mathbb{Z}^d\setminus\tilde{\omega}_1}\\
&=\lim_{\Lambda\uparrow\mathbb{Z}^d}\sum_{i=1}^{2d}\sum_{u\in F_i}\sideset{}{'}\sum_{\partial\omega_1=\{x,u\}}\rho_{\Lambda}(\omega_1)\langle\sigma_u\sigma_{y^i}\rangle_{\Lambda\setminus\tilde{\omega}_1}^f\\
&=\lim_{\Lambda\uparrow\mathbb{Z}^d}\sum_{i=1}^{2d}\sum_{u\in F_i}\sideset{}{'}\sum_{\partial\omega_1=\{x,u\}}\rho_{\Lambda}(\omega_1)\sum_{\substack{\partial\omega_2=\{u,y^i\}\\\omega_2\cap\tilde{\omega}_1=\emptyset}}\rho_{\Lambda\setminus\tilde{\omega}_1}(\omega_2)\label{eq:midstep1}\\
&=\lim_{\Lambda\uparrow\mathbb{Z}^d}\sum_{i=1}^{2d}\sum_{u\in F_i}\sideset{}{'}\sum_{\partial\omega_1=\{x,u\}}\sum_{\substack{\partial\omega_2=\{u,y^i\}\\\omega_2\cap\tilde{\omega}_1=\emptyset}}\rho_{\Lambda}(\omega_1\circ\omega_2)\label{eq:midstep2}\\
&\leq\lim_{\Lambda\uparrow\mathbb{Z}^d}\sum_{i=1}^{2d}\sum_{u\in \partial\Lambda_{Mn}}\sideset{}{'}\sum_{\partial\omega_1=\{x,u\}}\sum_{\substack{\partial\omega_2=\{u,y^i\}\\\omega_2\cap\tilde{\omega}_1=\emptyset}}\rho_{\Lambda}(\omega_1\circ\omega_2)\label{eq:midstep3}\\
&=\lim_{\Lambda\uparrow\mathbb{Z}^d}\sum_{i=1}^{2d}\sum_{\partial\omega=\{x,y^i\}}\rho_{\Lambda}(\omega)\\
&=\lim_{\Lambda\uparrow\mathbb{Z}^d}\sum_{i=1}^{2d}\langle\sigma_x\sigma_{y^i}\rangle_{\Lambda}^f=\sum_{i=1}^{2d}\langle\sigma_x\sigma_{y^i}\rangle_{\mathbb{Z}^d},\label{eq:sumrholast}
\end{align}
where we have applied Proposition \ref{prop:backbone} in \eqref{eq:midstep1} and \eqref{eq:midstep2} and in the last line, and where \eqref{eq:midstep3} follows trivially from the fact that $\rho_{\Lambda}(\omega_1\circ\omega_2)\geq0$. Combining \eqref{eq:sumrholast}, \eqref{eq:bb1} and \eqref{eq:cordif1} and letting $\Lambda\uparrow\mathbb{Z}^d$, we get

\begin{equation}\label{eq:tfin}
	\langle\sigma_x\sigma_y\rangle_{\Lambda_{Mn}}^f\geq\langle\sigma_x\sigma_y\rangle_{\mathbb{Z}^d}-\sum_{i=1}^{2d}\langle\sigma_x\sigma_{y^i}\rangle_{\mathbb{Z}^d}.
\end{equation}

We remark that even though we assume $\beta=\beta_c$ in the proof, it is not hard to see that \eqref{eq:tfin} actually holds for all $\beta\geq0$. Applying our assumption \eqref{eq:assumption} to \eqref{eq:tfin}, we get
\begin{equation}\label{eq:tfin+assumption}
\langle\sigma_x\sigma_y\rangle_{\Lambda_{Mn}}^f\geq c|x-y|^{-(d-2)}-2dC\left(2(M-1)n\right)^{-(d-2)}
\end{equation}
since $|x-y^i|\geq2(M-1)n$ for each $i$. This completes the proof of the theorem by choosing $M$ large and noting that $|x-y|\leq 2\sqrt{d}n$.
\end{proof}

Finally, we prove Theorem \ref{thm2}.
\begin{proof}[Proof of Theorem \ref{thm2}]
	For each fixed $f\in\mathcal{H}^3(\Lambda_L)$ and varying $a$, one obtains a
uniform exponential moment bound for the sequence of random variables
$\Phi_{\Lambda_L}^{a,0}(f)$ by the GHS inequality~\cite{GHS70} (see the proof of
Proposition 3.5 in \cite{CGN15}). This and our assumptions in the theorem imply
that for any $t\in\mathbb{R}$
	\begin{equation}\label{eq:zerofieldlimit}
		\lim_{a\downarrow0}\left\langle \exp[t\Phi^{a,0}_{\Lambda_L}(f)]\right\rangle_{\Lambda_L^a}^f=
		\exp[t^2\text{Var}(\Phi_{\Lambda_L}(f))/2] ,
	\end{equation}
(noting that the mean of $\Phi_{\Lambda_L}^{a,0}(f)$ is always $0$) where
\begin{equation}\label{eq:zerofieldvar}
	\text{Var}(\Phi_{\Lambda_L}(f)):=\int_{\Lambda_L}\int_{\Lambda_L}f(z)f(w)G_{\Lambda_L}(z,w)dzdw.
\end{equation}

For the near-critical magnetization field, we have
\begin{align}
	&\left\langle \exp[t\Phi^{a,h}_{\Lambda_L}(f)]\right\rangle_{\Lambda_L^a,H}^f=\left\langle \exp[ta^{(d+2)/2}\sum_{x\in\Lambda_L^a}f(x)\sigma_x]\right\rangle_{\Lambda_L^a,H}^f\\
	&\quad=\frac{\sum_{\sigma}\exp\left[\beta_c\sum_{\{x,y\}}\sigma_x\sigma_y+ha^{(d+2)/2}\sum_{x\in\Lambda_L^a}\sigma_x+ta^{(d+2)/2}\sum_{x\in\Lambda_L^a}f(x)\sigma_x\right]}{\sum_{\sigma}\exp\left[\beta_c\sum_{\{x,y\}}\sigma_x\sigma_y+ha^{(d+2)/2}\sum_{x\in\Lambda_L^a}\sigma_x\right]}\\
	&\quad=\frac{\left\langle \exp[\Phi^{a,0}_{\Lambda_L}(f_1)]\right\rangle_{\Lambda_L^a}^f}{\left\langle \exp[\Phi^{a,0}_{\Lambda_L}(f_2)]\right\rangle_{\Lambda_L^a}^f}, \label{eq:hfield1}
\end{align}
where the last equality follows from the second by dividing both the numerator and
denominator by the partition function $Z_{\Lambda_L^a}^f$ after setting
$f_1(z):=h+tf(z)$ and $f_2(z):=h$ for each $z\in\Lambda_L$. Applying
\eqref{eq:zerofieldlimit} and \eqref{eq:zerofieldvar} to \eqref{eq:hfield1}, we get
\begin{align}
	\lim_{a\downarrow0}\left\langle \exp[t\Phi^{a,h}_{\Lambda_L}(f)]\right\rangle_{\Lambda_L^a,H}^f&=\frac{ \exp[\text{Var}(\Phi_{\Lambda_L}(f_1))/2]}{ \exp[\text{Var}(\Phi_{\Lambda_L}(f_2))/2]}\\
	&=\exp\left[th\int_{\Lambda_L}\int_{\Lambda_L}f(z)G_{\Lambda_L}(z,w)dzdw+t^2\text{Var}(\Phi_{\Lambda_L}(f))/2\right],
\end{align}
which completes the proof of the first part of the theorem since the last displayed expression is the moment generating function for the claimed Gaussian distribution.

To prove the second part, let $M^a_{\Lambda_L}(\sigma) := a^{(d+2)/2} \sum_{x \in \Lambda^a_L} \sigma_x$. Since $\Phi_{\Lambda_L}^h$ is obtained by taking the scaling limit of the Ising model with $\beta=\beta_c$ and external field $H=ha^{(d+2)/2}$, we see that for each fixed $\sigma\in\{-1,+1\}^{\Lambda_L^a}$
\begin{align}
\mathbb{P}^f_{\Lambda^a_L,H}(\sigma) & =  \frac{\exp\left[\beta_c\sum_{\{x,y\}}\sigma_x\sigma_y+H\sum_{x\in\Lambda^a_L}\sigma_x\right]}{Z^f_{\Lambda^a_L,H}} \\
& =  \frac{Z^f_{\Lambda^a_L,0}}{Z^f_{\Lambda^a_L,H}} \mathbb{P}^f_{\Lambda^a_L,0}(\sigma) \exp\left[h M^a_{\Lambda_L}(\sigma)\right]\\
& =  \frac{\exp\left[h M^a_{\Lambda_L}(\sigma)\right]}{\left\langle \exp\left[h M^a_{\Lambda_L}(\sigma)\right] \right\rangle^f_{\Lambda^a_L,0}} \mathbb{P}^f_{\Lambda^a_L,0}(\sigma)
\end{align}
or
\begin{equation} \label{eq:Radon-Nikodym}
\frac{d\mathbb{P}^f_{\Lambda^a_L,H}}{d\mathbb{P}^f_{\Lambda^a_L,0}}(\sigma) = \frac{\exp\left[h M^a_{\Lambda_L}(\sigma)\right]}{\left\langle \exp\left[h M^a_{\Lambda_L}(\sigma)\right] \right\rangle^f_{\Lambda^a_L,0}}.
\end{equation}
Taking the scaling limit, and applying \eqref{eq:zerofieldlimit}  with $t=h$ and $f=1[\Lambda_L]$,
the Radon-Nikodym derivative \eqref{eq:Radon-Nikodym} converges weakly to
\begin{equation}
\exp\left[h M_{\Lambda_L} - \frac{h^2}{2}\text{Var}(M_{\Lambda_L})\right],
\end{equation}
where we have used the continuous mapping theorem and the fact that $M^a_{\Lambda_L}$ converges weakly to $M_{\Lambda_L}$.

\end{proof}

\section*{Acknowledgements}
The research of the second author was partially supported by NSFC grant 11901394 and that of the third author by US-NSF grant DMS-1507019. The authors thank Akira Sakai and Gordon Slade for useful comments.

\bibliographystyle{abbrv}
\bibliography{reference}

\end{document}